\documentclass[11pt]{amsart}
\usepackage[letterpaper,body={13.6cm,18.5cm}, mag=1000]{geometry}
\usepackage{amssymb}
\usepackage{amsthm}
\usepackage{amscd}
\newcommand{\para}{\par\vspace{.25cm}}
\newcommand{\Para}{\par\vspace{.5cm}\noindent}

\numberwithin{equation}{section}
\theoremstyle{plain}
\newtheorem{cor}[equation]{Corollary}
\newtheorem{lemma}[equation]{Lemma}

\newtheorem*{thma}{Theorem 1}
\newtheorem*{thmb}{Theorem 2}

\newtheorem*{thmd}{Theorem 4}

\newtheorem*{thmf}{Theorem 8}
\newtheorem*{thmg}{Theorem 7}

\newtheorem*{cora}{Corollary 3}
\newtheorem*{corb}{Corollary 5}
\newtheorem*{corc}{Corollary 6}
\newtheorem*{lemma1'}{Lemma 2.2$'$}
\theoremstyle{definition}

\newtheorem{remark}[equation]{Remark}

\newcommand{\dlabel}[1]{\ifmmode \text{\ttfamily \upshape [#1] } \else
{\ttfamily \upshape [#1] }\fi \label{#1}}

\newcommand{\B}{\operatorname{B} }
\newcommand{\C}{\operatorname{C} }

\newcommand{\Ha}{\operatorname{H} }

\newcommand{\Z}{\operatorname{Z} }

\newcommand{\Aut}{\operatorname{Aut} }

\newcommand{\Inn}{\operatorname{Inn} }

\newcommand{\Autcent}{\operatorname{Autcent} }

\begin{document}

\title{Automorphisms of abelian group extensions}

\author{I. B. S. Passi}
\address{Center for Advanced Study in Mathematics, Panjab University, Chandigarh 160014, INDIA\\
and Indian Institute of Science Education and Research, Mohali, Transit Campus: MGSIPAP Complex, Sector 26, Chandigarh  160019, INDIA.}
\email{ibspassi@yahoo.co.in}

\author{Mahender Singh}
\address{The Institute of Mathematical Sciences, CIT Campus, Taramani, Chennai 600113, INDIA.}
\email{mahender@imsc.res.in}

\author{Manoj K. Yadav}
\address{School of Mathematics, Harish-Chandra Research Institute, Chhatnag Road, Jhunsi, Allahabad 211019, INDIA.}
\email{myadav@hri.res.in}


\begin{abstract}
Let $1 \to N \to G \to H \to 1$ be an abelian extension. The purpose of this paper is to study the problem of  extending automorphisms of $N$ and  lifting automorphisms of $H$ to certain automorphisms of $G$.
\end{abstract}

\maketitle

\section{Introduction}
Let
$1 \to N \to G \stackrel{\pi}{\to} H \to 1$
be a short exact sequence of groups, i.e., an extension of a group $N$ by the group $H \simeq G/N$. If $N$ is abelian, then such an extension is called an \textit{abelian extension}. Our  aim in this paper  is to construct certain exact sequences, similar to the  one due to C. Wells \cite{Wells},   and apply  them to study extensions and liftings of automorphisms in abelian extensions. More precisely, we study, for abelian extensions, the following well-known problem  (see \cite{buckley75, Jin, robinson81}):

\par\vspace{.5cm}
\noindent {\bf Problem.} {\it Let $N$ be a normal subgroup of $G$. Under what conditions (i) can an automorphism of $N$  be  extended to an automorphism of $G$; (ii) an automorphism of $G/N$  is induced by an automorphism of $G$?}

\par\vspace{.5cm} Let $1 \to N \to G \stackrel{\pi}{\to} H \to 1$ be an abelian extension. 
We fix a left transversal $t : H \to G$ for $H$ in $G$ such that $t(1) = 1$, so that every element of $G$ can be written uniquely as $t(x)n$ for some $x \in H$ and $n \in N$. 
Given an  element $x \in H$, we  define an action of $x$ on $N$ by setting $n^{x} = t(x)^{-1}nt(x)$; we thus have a homomorphism  $\alpha:H\to \Aut(N)$ enabling us to view $N$ as a right $H$-module.
\para
A pair $(\theta,\, \phi) \in \Aut(N) \times \Aut(H)$ is called \textit{compatible} if
${{\theta}}^{-1}x^{{\alpha}}{\theta} = (x^{\phi})^{{\alpha}}$ for all $x\in H$. Let $\C$ denote the group of all compatible pairs. Let  $\C_1=\{\theta\in \Aut(N)\,|\, (\theta,\,1)\in \C\}$ and $\C_2=\{\phi\in  \Aut(H)\,|\,(1,\,\phi)\in \C\}$.
\para
We denote by $\Aut^{N,\, H}(G)$ the group of all automorphisms of $G$ which centralize $N$ (i.e., fix $N$ element-wise) and induce identity on $H$. By $\Aut^N(G)$ and $\Aut_N(G)$ we denote respectively the group of all automorphisms of $G$ which centralize $N$ and the group of all automorphisms $\alpha$ of $G$ which normalize $N$ (i.e., $\alpha(N) = N$). By $\Aut_N^H(G)$ we denote the group of all automorphisms of $G$ which normalize $N$ and induce identity on $H$. 
\para
Observe  that an automorphism $\gamma \in \Aut_N(G)$ induces automorphisms $\theta \in \Aut(N)$ and $\phi \in \Aut(H)$ given by $\theta(n)= \gamma(n)$ for all $n \in N$ and $\phi(x)= \gamma(t(x))N$ for all $x \in H$. We can thus define a homomorphism $\tau: \Aut_N(G) \to \Aut(N) \times \Aut(H)$ by setting $\tau(\gamma)= (\theta,\, \phi)$. We denote the restrictions of $\tau$ to $\Aut_N^H(G)$ and $\Aut^N(G)$ by $\tau_1$ and $\tau_2$ respectively.
\para
With the above notation, there exists the following exact sequence, first constructed by C.  Wells \cite{Wells}, which relates automorphisms of group extensions with group cohomology:
\[1 \to \Z_{\alpha}^1(H,\, \Z(N)) \to \Aut_N(G)  \stackrel{\tau}{\to} C \to \Ha_{\alpha}^2(H,\, \Z(N)).\]
Recently P. Jin \cite{Jin} gave an explicit description of this sequence for automorphisms of $G$ acting trivially on $H$ and obtained some interesting results regarding extensions of automorphisms of $N$ to automorphisms of $G$ inducing identity on $H$. We continue in the present work this line of investigation.
\par\vspace{.5cm}

In Section 2, we establish our exact sequences.
\para
\begin{thma}\label{theorem1}
If $1 \to N \to G \stackrel{\pi}{\to} H \to 1$ is an abelian extension, then there  exist  the following two exact sequences:
\begin{equation}\label{thma1}
1 \to {\Aut}^{N,\, H}(G) \to \Aut_N^H(G) \stackrel{\tau_1}{\to} \C_1  \stackrel{\lambda_1}{\to} \Ha^2(H,\, N)
\end{equation}
and
\begin{equation}\label{thma2}
1 \to  \Aut^{N,\, H}(G) \to \Aut^N(G) \stackrel{\tau_2}{\to} \C_2 \stackrel{\lambda_2}{\to} \Ha^2(H,\, N).
\end{equation}
\end{thma}

For the definitions of maps $\lambda_1$ and $\lambda_2$, see \eqref{lambda1} and \eqref{lambda2}. It may be noted that these maps are not necessarily
homomorphisms (see Remark \ref{remark}).
\par\vspace{.5cm}

We say that an  extension $1 \to N \to G \stackrel{\pi}{\to} H \to 1$ is \textit{central} if $N \leq Z(G)$, the centre of $G$; for such extensions we construct a more general exact sequence.

\par\vspace{.25cm}

\begin{thmb}
If  $1 \to N \to G \stackrel{\pi}{\to} H \to 1$ is a central extension, then there exists  an exact sequence
\begin{equation}\label{thmb1}
1 \to \Aut^{N,\, H}(G) \to \Aut_N(G) \stackrel{\tau}{\to} \Aut(N) \times \Aut(H) \stackrel{\lambda}{\to} \Ha^2(H,\, N).
\end{equation}
\end{thmb}

\par\vspace{.25cm}

As a consequence of Theorem 1, we readily get the following result:

\par\vspace{.25cm}
\begin{cora}\label{cor1}
Let $N$ be an abelian normal subgroup of $G$ with $\Ha^2(G/N,\, N)$ trivial. Then
\begin{quote}
\begin{enumerate}
\item every element of $\C_1$ can be extended to an automorphism of $G$ centralizing $H$;
\item every element of $\C_2$ can be lifted to an automorphism of $G$ centralizing $N$.
\end{enumerate}\end{quote}
\end{cora}

\par\vspace{.25cm}\noindent
Clearly, if $G$ is finite and the map $x \mapsto x^{|G/N|}$ is an isomorphism of $N$, then $\Ha^2(G/N,\, N) = 1$. In particular, for the class of finite groups $G$ such that $|G/N|$ is coprime to $|N|$, we have $\Ha^2(G/N,\, N) = 1$ and hence Corollary 3 holds true for this class of groups.

\par\vspace{.25cm}

In Section 3, we apply Theorem 1 to reduce the problem of lifting of automorphisms of $H$ to $G$ to the problem of lifting of automorphisms of Sylow subgroups of $H$ to automorphisms of their pre-images in $G$, and prove the following result:

\par\vspace{.25cm}

\begin{thmd}\label{theorem4}
Let $N$ be an abelian normal subgroup of a finite  group $G$. Then the following hold:
\begin{enumerate}
\item An automorphism $\phi$ of $G/N$ lifts to an automorphism of $G$ centralizing $N$
provided the restriction of $\phi$ to some Sylow $p$-subgroup $P/N$ of $G/N$, for each prime number $p$ dividing $|G/N|$, lifts to an automorphism of $P$ centralizing $N$.
\item If an automorphism $\phi$ of $G/N$ lifts to an automorphism of $G$ centralizing $N$, then the restriction of $\phi$ to a characteristic subgroup $P/N$ of $G/N$ lifts to an automorphism of $P$ centralizing $N$.
\end{enumerate}
\end{thmd}
\par\vspace{.25cm}

We mention below two corollaries to illustrate Theorem 4. These corollaries show that, in many cases, the hypothesis of Theorem 4 is naturally satisfied.

\begin{corb}
 Let $N$ be an abelian normal subgroup of a finite  group $G$ such that $G/N$ is nilpotent. Then an automorphism $\phi$ of $G/N$ lifts to an automorphism of $G$ centralizing $N$ if, and only if the restriction of $\phi$ to each Sylow subgroup $P/N$ of $G/N$ lifts to an automorphism of $P$ centralizing $N$.
\end{corb}

An automorphism $\alpha$ of a group $G$ is said to be \emph{commuting automorphism} if $x \alpha(x) = \alpha(x) x$ for all $x \in G$. It follows from \cite[Remark 4.2]{DS02} that each Sylow subgroup of a finite group $G$ is kept invariant by every commuting automorphism of $G$. Thus we have the following result:

\begin{corc}
  Let $N$ be an abelian normal subgroup of a finite  group $G$. Then a commuting automorphism $\phi$ of $G/N$ lifts to an automorphism of $G$ centralizing $N$ if, and only if the restriction of $\phi$ to each Sylow subgroup $P/N$ of $G/N$ lifts to an automorphism of $P$ centralizing $N$.
\end{corc}

\par\vspace{.25cm}

It may be remarked that if $N$ is an abelian subgroup of $G$, then the map $\omega$ constructed by Jin in \cite[Theorem A]{Jin} is trivial and hence, in our notation, gives the following exact sequence:

\[1 \to \Z_{\alpha}^1(H,\, N) \to \Aut_N^H(G) \to \C_1 \to 1\]
which, in turn, implies that any element of $\C_1$ can be extended to an automorphism of $G$ centralizing $H$. The latter statement,  however,  is not true, in general, as illustrated by the following class of examples:
\par
\vspace{.5cm}
\noindent Let $G_1$ be a finite $p$-group, where $p$ is an odd prime. Then there exists a group $G$ in the isoclinism class (in the sense of P. Hall \cite{Hall}) of $G_1$ such that $G$ has no non-trivial abelian direct factor. Let $N= Z(G)$, the centre of $G$.  Then it follows from \cite[Corollary 2]{Adney} that $|\Aut^H(G)|$ is $p^r$ for some $r \geq 1$. Let $\theta \in N$ be the automorphism inverting elements of $N$. Then the order of $\theta$ is 2 and therefore $\theta$ cannot be extended to an automorphism of $G$ centralizing $H$. Thus Theorem D of Jin does not shed any light in case $N$ is abelian. However, using our sequence (\ref{thma1}), one can see that the following result holds: 
\par
\vspace{.25cm}

\begin{thmg}
Let $N$ be an abelian normal subgroup of a finite group $G$. Then an automorphism $\theta$ of $N$ extends to an automorphism of $G$ centralizing $G/N$ if, and only if, for some Sylow $p$-subgroup $P/N$ of $G/N$, for each prime number $p$ dividing $|G/N|$,  $\theta$ extends to an automorphism of $P$ centralizing $P/N$.
\end{thmg}

\par\vspace{.25cm}

Finally, in Section 4, we refine our sequences \eqref{thma1} -  \eqref{thmb1}, and show that these sequences split in case the given exact sequence $1\to N\to G\to H\to 1$ splits  (Theorem 8). We also give examples to show that the converse is not true, in general.

\par\vspace{.5cm}

\section{Construction of sequences}

\par\vspace{.5cm}

Let $1 \to N \to G \stackrel{\pi}{\to} H \to 1$ be an abelian extension.
For any two elements $x,\, y \in H$, we have $\pi(t(xy))= xy = \pi(t(x))\pi(t(y))= \pi(t(x)t(y))$. Thus there exists a unique element (say) $\mu(x,\, y) \in N$ such that $t(xy)\mu(x,\,y)= t(x)t(y)$. Observe that  $\mu$ is a map from
$H \times H$ to $N$ such that $\mu(1,\,x) = \mu(x,\, 1) = 1$ and 
\begin{equation}\label{eqn1}
\mu(xy,\, z) \mu(x,\, y)^{t(z)} = \mu(x,\, yz) \mu(y,\, z),
\end{equation}
for all $x,\, y,\, z \in H$; in other words, $\mu: H\times H\to N$ is a normalized $2$-cocycle.
\para
We begin by recalling  a result of Wells \cite{Wells}; since we are dealing with abelian extensions, the proof of this  result in the present  case is quite easy. However, for the reader's convenience, we include a proof here.
\para
\begin{lemma}\label{lemma1}\cite{Wells}
Let $1 \to N \to G \stackrel{\pi}{\to} H \to 1$ be an abelian extension. If $\gamma \in \Aut_N(G)$, then there is a triplet $(\theta,\, \phi,\, \chi) \in \Aut(N)\times \Aut(H) \times N^H$ such that for all $x,\,~y \in H$ and $n \in N$ the following conditions are satisfied:
\begin{quote}
\begin{enumerate}
\item $\gamma(t(x)n) = t(\phi(x)) \chi(x) \theta(n),$\\
\item $\mu(\phi(x),\, \phi(y)) \theta(\mu(x,\, y)^{-1}) = (\chi(x)^{-1})^{t(\phi(y))} \chi(y)^{-1} \chi(xy),$\\
\item $\theta(n^{t(x)})=\theta(n)^{t(\phi(x))}.$
\end{enumerate}\end{quote}
$[$\,Here $N^H$ denotes the group of all maps $\psi$ from $H$ to $N$ such that $\psi(1) = 1$.\,$]$
\para
Conversely, if $(\theta,\, \phi,\, \chi) \in \Aut(N)\times \Aut(H) \times N^H$ is a triplet satisfying equations $(2)$ and $(3)$, then $\gamma $ defined by $(1)$ is an automorphism of $G$ normalizing $N$.
\end{lemma}

\begin{proof}
Every automorphism $\gamma \in \Aut_N(G)$, determines a pair $(\theta,\, \phi) \in \Aut(N) \times \Aut(H)$ such that $\gamma$ restricts to $\theta$ on $N$ and induces $\phi$ on $H$. For any $x \in H$, we have $\pi(\gamma(t(x))) = \phi(x)$. Thus 
\begin{equation}\label{eqn2}
\gamma(t(x)) = t(\phi(x)) \chi(x),
\end{equation}
for some element $\chi(x) \in N$. Since $\chi(x)$ is unique for a given $x \in H$, it follows that $\chi$ is a map from $H$ to $N$. Notice that $\chi(1) = 1$. Let $g \in G$. Then $g = t(x)n$ for some $x \in H$ and $n \in N$. Applying $\gamma$, we have 
\begin{equation}\label{eqn3}
\gamma(g) = \gamma(t(x)) \theta(n) = t(\phi(x)) \chi(x) \theta(n).
\end{equation}
\para
Let $x,\, y \in H$. Then $t(xy)\mu(x,\, y)= t(x)t(y)$. On applying $\gamma$ we get $\gamma(t(xy)) \theta(\mu(x,\, y))= \gamma(t(x)) \gamma(t(y)) $, since $\gamma$ restricts to $\theta$ on $N$. By \eqref{eqn2}, we have $\gamma(t(x)) = t(\phi(x)) \chi(x)$, $\gamma(t(y)) = t(\phi(y)) \chi(y)$ and $\gamma(t(xy)) = t(\phi(xy)) \chi(xy)$, and consequently
\[ t(\phi(xy)) \chi(xy)\theta(\mu(x,\, y)) =  t(\phi(x)) \chi(x) t(\phi(y)) \chi(y).\]
This, in turn, gives
\begin{equation}\label{eqn4}
\mu(\phi(x),\, \phi(y)) \theta(\mu(x,\, y)^{-1}) = (\chi(x)^{-1})^{t(\phi(y))} \chi(y)^{-1} \chi(xy).
\end{equation}
For $x \in H$ and $n \in N$, we have
\begin{eqnarray}
\theta(n^{t(x)}) &=& \gamma(n^{t(x)}) = \theta(n)^{\gamma(t(x))}\nonumber\\
&=& \theta(n)^{t(\phi(x))\chi(x)} = \theta(n)^{t(\phi(x))}\label{eqn5},
\end{eqnarray}
since $\theta$ is the restriction of $\gamma$ and $\chi(x)$ commutes with $\theta(n)$.
Thus, given an element $\gamma \in \Aut_N(G)$, there is a triplet $(\theta,\, \phi,\, \chi) \in \Aut(N)\times \Aut(H) \times N^H$ satisfying equations (1),\, (2) and (3).
\para
Conversely, let $(\theta,\, \phi,\, \chi) \in \Aut(N)\times \Aut(H) \times N^H$ be a triplet satisfying equations (2) and (3).  We proceed to verify that $\gamma $ defined by (1) is an automorphism of $G$ normalizing $N$. Let $g_1= t(x_1)n_1$ and
$g_2= t(x_2)n_2$ be elements of $G$, where $x_1,\,~x_2 \in H$ and $n_1,\,~n_2 \in N$. Then
{\setlength\arraycolsep{2pt}
\begin{eqnarray*}
\gamma(g_1g_2) &=& \gamma(t(x_1)n_1t(x_2)n_2)\nonumber\\
&=& \gamma(t(x_1x_2)\mu(x_1,\,x_2)n_1^{t(x_2)}n_2)\nonumber\\
&=& t(\phi(x_1x_2))\chi(x_1x_2)\theta(\mu(x_1,\,x_2)n_1^{t(x_2)}n_2)\nonumber\\
&=& t(\phi(x_1x_2))\mu(\phi(x_1),\, \phi(x_2))\chi(x_1)^{t(\phi(x_2))}\chi(x_2)\theta(n_1)^{t(\phi(x_2))}\theta(n_2) \nonumber\\
&=& t(\phi(x_1))\chi(x_1)\theta(n_1)t(\phi(x_2))\chi(x_2)\theta(n_2)\nonumber\\
&=& \gamma(g_1) \gamma(g_2).
\end{eqnarray*}}
\noindent Hence $\gamma$ is a homomorphism. Let $g= t(x)n$ be an element of $G$. Since $\phi$ and $\theta$ are onto, there exist elements $x' \in H$ and $n' \in N$ such that $\phi(x')=x$ and $\theta(n')=n$. We then have $\gamma(t(x')\theta^{-1}(\chi(x')^{-1})n')= g$. Hence $\gamma$ is onto.
\para Finally, let $\gamma(t(x)n)=1$. Then $t(\phi(x)) \in N$, and it easily follows that  $t(x)n=1$; consequently $\gamma$ is one-one. Also $\gamma(n)=\theta(n)$ for $n \in N$. Therefore $\gamma \in \Aut_N(G)$.  $\Box$
\end{proof}
\begin{remark}
If  $1 \to N \to G \stackrel{\pi}{\to} H \to 1$ is a central extension, then the action of $H$ on $N$ becomes trivial, and therefore Lemma \ref{lemma1} takes the following simpler form which we will use in the proof of Theorem 2.
\end{remark}
\para
\noindent {\bf Lemma 2.2$'$.} \textit{Let $1 \to N \to G \stackrel{\pi}{\to} H \to 1$ be a central extension. If $\gamma \in \Aut_N(G)$, then there exists a triplet $(\theta,\, \phi,\, \chi) \in \Aut(N)\times \Aut(H) \times N^H$ such that for all $x,\,~y \in H$ and $n \in N$ the following conditions are satisfied:}
\vspace{.8mm}
\begin{quote}
(1$'$) $\gamma(t(x)n) = t(\phi(x)) \chi(x) \theta(n),$\\
\\
(2$'$) $\mu(\phi(x),\, \phi(y)) \theta(\mu(x,\, y)^{-1}) = \chi(x)^{-1} \chi(y)^{-1} \chi(xy).$
\end{quote}
\para
\textit{Conversely, if $(\theta,\, \phi,\, \chi) \in \Aut(N)\times \Aut(H) \times N^H$ is a triplet satisfying equation} (2$'$),\, \textit{then $\gamma $ defined by} (1$'$) \textit{is an automorphism of $G$ normalizing $N$.}
\par\vspace{3mm}

For $\theta \in \C_1$ and $\phi \in \C_2$, we define maps $k_{\theta},~ k_{\phi}:H \times H \to N$ by setting, for $x,\,~y \in H$, 
$$k_{\theta}(x,\,y)= \mu(x,\, y) \theta(\mu(x,\, y)^{-1})$$
 and
 $$k_{\phi}(x,\,y)= \mu(\phi(x),\, \phi(y))\mu(x,\, y)^{-1}.$$ 
Notice that, for $\theta \in \C_1$, we have $\theta(n^{t(x)})=\theta(n)^{t(x)}$ for all $x \in H$ and $n \in N$. Similarly, for $\phi \in \C_2$, we have $n^{t(x)}= n^{t(\phi(x))}$ for all $x \in H$ and $n \in N $.

\par\vspace{.25cm}

\begin{lemma}\label{lemma2}
The maps $k_{\theta}$ and $k_{\phi}$ are normalized $2$-cocycles.
\end{lemma}
\para
\begin{proof}
For $x,\, ~y,\, ~z \in H$, we have
{\setlength\arraycolsep{2pt}
\begin{eqnarray*}
k_{\theta}(xy,\, z) k_{\theta}(x,\, y)^{t(z)} &=& \mu(xy,\, z) \theta(\mu(xy,\, z)^{-1}) (\mu(x,\, y) \theta(\mu(x,\, y)^{-1}))^{t(z)}\\
&=& \mu(xy,\, z)\mu(x,\, y)^{t(z)} \theta(\mu(xy,\, z)^{-1}(\mu(x,\, y)^{-1})^{t(z)})\\
&=& \mu(x,\, yz) \mu(y,\, z)\theta(\mu(x,\, yz)^{-1} \mu(y,\, z)^{-1})~~\textrm{by}~~ \eqref{eqn1}\\
&=& \mu(x,\, yz) \theta(\mu(x,\, yz)^{-1}) \mu(y,\, z) \theta(\mu(y,\, z)^{-1})\\
&=& k_{\theta}(x,\, yz)k_{\theta}(y,\, z),
\end{eqnarray*}}
 and $k_{\theta}(x,\,1)=1= k_{\theta}(1,\,x)$. Hence $k_{\theta}$ is a normalized $2$-cocycle.
\para
 We next show that $k_{\phi}$ is a normalized $2$-cocycle. For  $x,\, ~y,\, ~z \in H$, we have
{\setlength\arraycolsep{2pt}
\begin{eqnarray*}
k_{\phi}(xy,\, z) k_{\phi}(x,\, y)^{t(z)} &=& \mu(\phi(xy),\, \phi(z)) \mu(xy,\, z)^{-1} (\mu(\phi(x),\, \phi(y)) \mu(x,\, y)^{-1})^{t(z)}\\
&=& \mu(\phi(xy),\, \phi(z))\mu(\phi(x),\, \phi(y))^{t(z)}\mu(xy,\, z)^{-1}(\mu(x,\, y)^{-1})^{t(z)}\\
&=& \mu(\phi(x),\, \phi(yz)) \mu(\phi(y),\, \phi(z)) \mu(x,\, yz)^{-1} \mu(y,\, z)^{-1}~~\textrm{by}~~ \eqref{eqn1} \\
&=& \mu(\phi(x),\, \phi(yz))\mu(x,\, yz)^{-1} \mu(\phi(y),\, \phi(z))\mu(y,\, z)^{-1}\\
&=& k_{\phi}(x,\, yz) k_{\phi}(y,\, z),
\end{eqnarray*}}
and $k_{\phi}(x,\,1)=1= k_{\phi}(1,\,x)$. Thus  the map $k_{\phi}$ is a normalized 2-cocycle. This completes the proof of the lemma.  $\Box$
\end{proof}
\para
Define $\lambda_1:\C_1 \to \Ha^2(H,\, N)$ by setting, for $\theta\in \C_1$, 
\begin{equation}\label{lambda1}
\lambda_1(\theta)= [k_{\theta}],\,~ \textrm{the cohomology class of}~ k_{\theta};
\end{equation}
similarly, define $\lambda_2:\C_2 \to \Ha^2(H,\, N)$ by setting, for $\phi\in \C_2$, 
\begin{equation}\label{lambda2}
\lambda_2(\phi)= [k_{\phi}],\,~ \textrm{the cohomology class of}~ k_{\phi}.
\end{equation}

To justify this definition, we need the following:
\para
\begin{lemma}\label{lemma3}
The maps $\lambda_1$ and $\lambda_2$ are well-defined.
\end{lemma}
\par
\begin{proof}
To show that the maps $\lambda_1$ and $\lambda_2$ are well-defined, we need to show that these maps are independent of the choice of transversals. Let $t, ~s:H \to N$ be two transversals with $t(1)=1=s(1)$. Then there exist maps $\mu, ~\nu:H \times H \to N$ such that for $x,\, y \in H$ we have $t(xy)\mu(x,\,y)=t(x)t(y)$ and $s(xy)\nu(x,\,y)= s(x)s(y)$. For $x \in H$, since $t(x)$ and $s(x)$ satisfy $\pi(t(x))=x=\pi(s(x))$, there exists a unique element (say) $\lambda(x) \in N$ such that $t(x)= s(x) \lambda(x)$. We thus have a map $\lambda: H \to N$ with $\lambda(1)= 1$. For $x,\,~ y \in H$,\, $t(xy)=s(xy)\lambda(xy)$. This gives $t(x)t(y)\mu(x,\,y)^{-1}= s(x)s(y)\nu(xy)^{-1}\lambda(xy)$. Putting $t(u)= s(u) \lambda(u)$, where $u=x, ~y$, we have $\lambda(x)^{s(y)}\lambda(y) \lambda(xy)^{-1}= \mu(x,\,y)\nu(x,\,y)^{-1}$. Since $\lambda(1)=1$,  $\mu(x,\,y)\nu(x,\,y)^{-1} \in \B^2(H,\, N)$, the group of 2-coboundaries.
\para
Similarly,
{\setlength\arraycolsep{2pt}
\begin{eqnarray*}
\theta(\mu(x,\,y))\theta(\nu(x,\,y)^{-1}) &=&  \theta( \mu(x,\,y)\nu(x,\,y)^{-1})\\
&=& \theta(\lambda(x)^{S(y)}\lambda(y)\lambda(xy)^{-1})\\
&=& \theta(\lambda(x)^{S(y)}) \theta(\lambda(y)) \theta(\lambda(xy)^{-1})\\
&=& \theta(\lambda(x))^{S(y)} \theta(\lambda(y)) \theta(\lambda(xy)^{-1})\\
&=& \lambda'(x)^{S(y)} \lambda'(y) \lambda'(xy)^{-1} \in \B^2(H,\, N),
\end{eqnarray*}}
where $\lambda' = \theta \lambda$. This proves that $\lambda_1$ is independent of the choice of a transversal.
\para
Next we prove that $\lambda_2$ is well-defined. It is sufficient to show that 
\[\mu(\phi(x),\, \phi(y))\nu(\phi(x),\, \phi(y))^{-1} \in \B^2(H,\, N).\]
Just as above, we have 
\[\lambda(\phi(x))^{s(\phi(y))}\lambda(\phi(y)) \lambda(\phi(xy))^{-1}= \mu(\phi(x),\, \phi(y))\nu(\phi(x),\, \phi(y))^{-1}.\] Putting $\lambda \phi= \lambda''$, we get 
\[\lambda''(x)^{S(\phi(y))}\lambda''(y) \lambda''(xy)^{-1}= \mu(\phi(x),\, \phi(y))\nu(\phi(x),\, \phi(y))^{-1}.\] Since $n^{S(x)}= n^{S(\phi(x))}$ and $\lambda''(1)=1$,  $\mu(\phi(x),\, \phi(y))\nu(\phi(x),\, \phi(y))^{-1} \in \B^2(H,\, N)$. This proves that $\lambda_2$ is also independent of the choice of a transversal, and the proof of the Lemma is complete.  $\Box$

\end{proof}
\para
\noindent \textbf{Proof of Theorem 1.} Let $1 \to N \to G \to H \to 1$ be an abelian extension. Clearly both the sequences \eqref{thma1} and \eqref{thma2} are exact at the first two terms. To complete the proof it only remains to show the exactness at the third term of the respective sequences.
\para
First consider \eqref{thma1}. Let $\gamma \in \Aut_N^H(G)$. Then $\theta \in \C_1$, where $\theta$ is the restriction of $\gamma$ to $N$. For $x,\,~y \in H$, we have, by Lemma \ref{lemma1}(2), $k_{\theta}(x,\, y) =  (\chi(x)^{-1})^{t(y)} \chi(y)^{-1} \chi(xy)$. Thus   $k_{\theta} \in \B^2(H,\, N)$ and hence $\lambda_1(\theta)= 1$. Conversely, if $\theta \in \C_1$ is such that  $\lambda_1(\theta)= 1$, then for $x,\,~y \in H$, we have  $k_{\theta}(x,\,y)=  (\chi(x)^{-1})^{t(y)}\chi(y)^{-1} \chi(xy)$,
where $\chi : H \to N$ with $\chi(1) = 1$. Therefore $\gamma$ defined by Lemma \ref{lemma1}(1) is an element of $\Aut_N^H(G)$. Hence the sequence \eqref{thma1} is exact.
\para
Next let us consider the sequence \eqref{thma2}. Let $\gamma \in \Aut^N(G)$. Then  $ \phi \in \C_2$, where $\phi$ is induced by $\gamma$ on $H$. For $x,\,~y \in H$, we have  $k_{\phi}(x,\, y) =  (\chi(x)^{-1})^{t(\phi(y))} \chi(y)^{-1} \chi(xy)$
by Lemma \ref{lemma1}(2). Since $n^{t(\phi(y))} = n^{t(y)}$ for all $n \in N$ and $y \in H$, we have $k_{\phi} \in \B^2(H,\, N)$ and hence $\lambda_2(\phi)= 1$. 
\para
Conversely, if $\phi \in \C_2$ is such that  $\lambda_2(\phi)=1$, then, for $x,\,~y \in H$, we have  $k_{\phi}(x,\,y)=  (\chi(x)^{-1})^{t(y)}\chi(y)^{-1} \chi(xy)$, where $\chi : H \to N$ is a map with $\chi(1) = 1$. Therefore $\gamma$ defined by Lemma \ref{lemma1}(1) is an element of $\Aut_H^N(G)$. Hence the sequence \eqref{thma2} is exact, and the proof of Theorem 1 is complete.   $\Box$
\Para
\noindent \textbf{Proof of Theorem 2.}
The sequence \eqref{thmb1} is clearly exact at $\Aut^{N,\, H}$ and $\Aut_N(G)$. We construct the map $\lambda$, and show the exactness at $\Aut(N) \times \Aut(H)$. For $(\theta,\, \phi) \in \Aut(N)\times \Aut(H)$, define $k_{\theta,\, \phi}: H \times H \to N$ by setting, for  $x,\,~y \in H$, 
$$k_{\theta,\, \phi}(x,\,y)= \mu(\phi(x),\, \phi(y)) \theta(\mu(x,\, y))^{-1}.$$
Observe  that for $x,\,~y,\,~z \in H$, we have $k_{\theta,\, \phi}(x,\,1)=1= k_{\theta,\, \phi}(1,\,x)$ and 
\[k_{\theta,\, \phi}(xy,\,z) k_{\theta,\, \phi}(x,\, y) = k_{\theta,\, \phi}(x,\, yz)k_{\theta,\, \phi}(y,\, z).\]
Thus  $k_{\theta,\, \phi} \in \Z^2(H,\, N)$, the group of normalized 2-cocycles. Define $\lambda(\theta,\, \phi)= [k_{\theta,\, \phi}]$, the cohomology class of $k_{\theta,\, \phi}$ in $H^2(H,\,N)$.  Proceeding as in the proof of Lemma \ref{lemma3}, one can prove that $\lambda$ is well-defined.
If $(\theta,\, \phi) \in \Aut(N)\times \Aut(H)$ is induced by some $\gamma \in \Aut_N(G)$, then by Lemma 2.2$'$, we have
$$k_{\theta,\, \phi}(x,\,y)= \chi(x)^{-1}\chi(y)^{-1}\chi(xy),$$
where $\chi:H \to N$ is a map with $\chi(1)=1$. Thus $k_{\theta,\, \phi}(x,\,y) \in \B^2(H,\, N)$. Hence $\lambda (\theta,\, \phi)= 1$.

\para Conversely,\, if $(\theta,\, \phi) \in \Aut(N)\times \Aut(H)$ is such that  $[k_{\theta,\, \phi}]= 1$, then $k_{\theta,\, \phi}(x,\,y)=\chi(x)^{-1} \chi(y)^{-1} \chi(xy)$ for some $\chi:H \to N$ with $\chi(1)=1$. By Lemma 2.2$'$ there exist $\gamma \in \Aut_N(G)$ inducing $\theta$ and $\phi$. Thus the sequence \eqref{thmb1} is exact.  $\Box$

\par\vspace{.25cm}
\begin{remark}\label{remark}
The maps $\lambda_1$ and $\lambda_2$ are not homomorphisms, but they turn out to be  derivations with respect to natural actions of $\C_1$ and $\C_2$ respectively on $\Ha^2(H,\, N)$. There is an action of $\C_1$ on $\Z^2(H,\, N)$ given by $(\theta,\, k) \mapsto k^{\theta}$ for $\theta \in \C_1$ and $k \in \Z^2(H,\, N)$, where $k^{\theta}(x,\,y)= \theta\big( k(x,\, y)\big)$
for $x,\,~ y \in H$. Note that if $k \in \B^2(H,\, N)$, then $k(x,\, y)=  \chi(x)^{t(y)}\chi(y) \chi(xy)^{-1}$. Since $\theta(\chi(x)^{t(y)})= \theta(\chi(x))^{t(y)}$, we have $$k^{\theta}(x,\,y)=  \theta(\chi(x))^{t(y)}\theta(\chi(y))\theta(\chi(xy))^{-1}.$$ Putting $\chi'= \theta \chi$, we have $k^{\theta}(x,\,y)= \chi'(x)^{t(y)}\chi'(y)\chi'(xy)^{-1}$. Thus the action keeps $\B^2(H,\, N)$ invariant and hence induces an action on $\Ha^2(H,\, N)$ given by $(\theta,\, [k]) \mapsto [k^{\theta}]$.
\para
One can see that for $\theta_1,\,~ \theta_2,\in \C_1$ and $x,\,~y \in H$, we have
{\setlength\arraycolsep{2pt}
\begin{eqnarray*}
k_{\theta_1\theta_2,\, 1}(x,\,y) &=&  \mu(x,\, y) \theta_1 \theta_2(\mu(x,\, y)^{-1})\\
&=& \mu(x,\, y) \theta_1(\mu(x,\, y)^{-1})\theta_1(\mu(x,\, y) \theta_2(\mu(x,\, y)^{-1}))\\
&=& k_{\theta_1,\, 1}(x,\,y)k_{\theta_2,\, 1}^{\theta_1}(x,\, y)
\end{eqnarray*}}
\noindent Thus $k_{\theta_1\theta_2,\, 1}= k_{\theta_1,\, 1}k_{\theta_2,\, 1}^{\theta_1}$ and hence $\lambda_1(\theta_1\theta_2)= \lambda_1(\theta_11)\lambda_1(\theta_2)^{\theta_1}$. Consequently  $\lambda_1$ is a derivation with respect to this action.
\para
Similarly, there is an action of $\C_2$ on $\Z^2(H,\, N)$ given by $(\phi,\, k) \mapsto k^{\phi}$ for $\phi \in \C_2$ and $k \in \Z^2(H,\, N)$, where $k^{\phi}(x,\,y)= k(\phi(x),\, \phi(y))$ for $x,\,~ y \in H$. If $k \in \B^2(H,\, N)$, then $k(x,\, y)= \chi(x)^{t(y)}\chi(y)$ $\chi(xy)^{-1}$. Since $\chi(\phi(x))^{t(\phi(y))}= \chi(\phi(x))^{t(y)}$, we have $$k^{\phi}(x,\,y)=  \chi(\phi(x))^{t(y)}\chi(\phi(y))\chi(\phi(xy))^{-1}.$$ Putting $\chi'=  \chi \phi$, we have $k^{\phi}(x,\,y)= \chi'(x)^{t(y)}\chi'(y)\chi'(xy)^{-1}$. Thus the action keeps $\B^2(H,\, N)$ invariant and hence induces an action on $\Ha^2(H,\, N)$ given by $(\phi,\, [k]) \mapsto [k^{\phi}]$. Just as above, one can see that $\lambda_2$ is a derivation with respect to this action.

\end{remark}

\par
\vspace{.5cm}

\section{Applications}

\par\vspace{.5cm}

In this section we give some applications of our exact sequences to lifting and extension of automorphisms in abelian extensions. 

\par\vspace{.5cm}
\noindent \textbf{Proof of Theorem 4.} (1) Let $N$ be an abelian normal subgroup of a finite  group $G$. 
Suppose that the restriction $\phi|_P$ of $\phi$ to any Sylow subgroup $P/N$ of $G/N$ lifts to an automorphism of $P$ centralizing $N$. Then the pair  $(1,\, \phi|_P)$ is compatible, and hence, as is easy to see,  $(1,\, \phi)$ is also compatible. Applying sequence \eqref{thma2} of Theorem 1 to the abelian extension $1 \to N \to P \to P/N \to 1$, we have that the cohomology class $[k_{\phi|_P}]=1$ in $\Ha^2(P/N,\, N)$. It follows from the construction of the cochain complex defining the group cohomology, that the map $\Ha^2(G/N,\, N) \to \Ha^2(P/N,\, N)$ induced by the inclusion $P/N \hookrightarrow G/N$ maps the class $[k_{\phi}]$ to $[k_{\phi|_P}]$. However,  by \cite[Chapter III,  Proposition 9.5 (ii)]{Brown}, we have $[G/N:P/N][k_{\phi}]=1$. Since this holds for at least one Sylow $p$-subgroups $P/N$ of $G/N$, for each prime number $p$ dividing $|G/N|$, it follows that $[k_{\phi}]=1$. By exactness of sequence \eqref{thma2} of Theorem 1, $\phi$ lifts to an automorphism $\gamma$ of $G$ centralizing $N$.
\para
\noindent(2) Let $\gamma$ be a lift of $\phi$ to an automorphism of $G$ centralizing $N$. To complete the proof  it only needs to be observed that if $P/N$ is a characteristic subgroup of $G/N$, then $P$ is invariant under $\gamma$. So the restriction of $\gamma$ to $P$ is the required lift.  
$\Box$
\para

The proof of Theorem 7 is similar to the above proof and we omit the details.

\Para
 For the case of central extensions, the sequence \eqref{thmb1} yields the following result:

\para\begin{cor}
Let $N$ be a central subgroup of a finite group $G$. Then a pair $(\theta,\, \phi) \in \Aut(N) \times \Aut(G/N)$ lifts to an automorphism of $G$ 
provided for some Sylow $p$-subgroup $P/N$ of $G/N$, for each prime number $p$ dividing $|G/N|$,\, $(\theta,\, \phi|_P) \in \Aut(N) \times \Aut(P/N)$ lifts to an automorphism of $P$.
\end{cor}
\para
\section{Splitting of sequences}

\par\vspace{.5cm}
Let $1 \to N \to G \to H \to 1$ be an abelian extension. Let $\C_1^*= \{ \theta\in \C_1\,|\, \lambda_1(\theta)= 1\}$ and $\C_2^*= \{ \phi \in \C_2\,|\, \lambda_2(\phi)= 1\}$. Then it follows from Theorem 1 that the sequences
\begin{equation}\label{eqn6}
1 \to {\Aut}^{N,\, H}(G) \to \Aut_N^H(G) \to \C_1^* \to 1
\end{equation}
and
\begin{equation}\label{eqn7}
1 \to  \Aut^{N,\, H}(G) \to \Aut^N(G) \to \C_2^* \to 1
\end{equation}
are exact.
\para
Similarly,\, let $1 \to N \to G \to H \to 1$ be a central extension and $\C^*= \{ (\theta,\, \phi)\in \Aut(N) \times \Aut(H) \,|\, \lambda(\theta,\, \phi)= 1\}$. Then it follows from Theorem 2 that the sequence
\begin{equation}\label{eqn8}
1 \to \Aut^{N,\, H}(G) \to \Aut_N(G) \to \C^* \to 1
\end{equation}
is exact.
\para

\begin{thmf}\label{theorem6}
Let $G$ be a finite group and $N$  an abelian normal subgroup of $G$ such that the sequence $1 \to N \to G \to H \to 1$ splits. Then the sequences \eqref{eqn6} and \eqref{eqn7} split. Further, if $N$ is a central subgroup of $G$, then the sequence \eqref{eqn8} also splits.
\end{thmf}
\par\vspace{.15cm}
\begin{proof}
Since the sequence $1 \to N \to G \to H \to 1$ splits, we have $\mu \in \B^2(H,\, N)$. We can write $G= N \rtimes H$ as a semidirect product of $N$ by $H$. Every element $g \in G$ can be written uniquely  as $g=hn$ with  $h \in H$ and $n \in N$. 
\para
We first  show that the sequence \eqref{eqn6} splits. Note that $\C_1^* = \{ \theta \in \Aut(N)\, |\, \theta(n^h)= \theta(n)^h~  \textrm{for all}~ n \in N~ \textrm{and}~ h \in H \}$. Define a map $\psi_1:\C_1^* \to \Aut_N^H(G)$ by $\psi_1(\theta)=\gamma_1$, where $\gamma_1:G \to G$ is given by $\gamma_1(g)=\gamma_1(hn)= h\theta(n)$ for $g=hn$ in $G$. Then for $g_1=h_1n_1$,\, $g_2=h_2n_2$ in $G$, we have
{\setlength\arraycolsep{2pt}
\begin{eqnarray*}
\gamma_1(g_1g_2) &=&  \gamma_1\big((h_1n_1)(h_2n_2)\big)= \gamma_1(h_1h_2n_1^{h_2}n_2)\\
&=& h_1h_2 \theta (n_1^{h_2}n_2)= h_1h_2\theta(n_1)^{h_2}\theta(n_2)\\
&=& \gamma_1(g_1)\gamma_1(g_2),
\end{eqnarray*}}
showing that $\gamma_1$ is a homomorphism. It is easy to see that $\gamma_1$ is an automorphism of $G$ which normalizes $N$ and induces identity on $H$. Notice that $\psi_1$ is a section in the sequence \eqref{eqn6} and hence the sequence splits.
\para
Next we show that  the sequence \eqref{eqn7} splits. Notice that $\C_2^* = \{ \phi \in \Aut(H) \,|\, n^{\phi(h)}= n^h~  \textrm{for all}~ n \in N~ \textrm{and}~ h \in H \}$. Define a map $\psi_2:\C_2^* \to \Aut^N(G)$ by setting $\psi_2(\phi)=\gamma_2$, where $\gamma_2:G \to G$ is given by $\gamma_2(g)=\gamma_2(hn)= \phi(h)n$ for $g=hn$ in $G$. Then for $g_1=h_1n_1$, $g_2=h_2n_2$ in $G$, we have
{\setlength\arraycolsep{2pt}
\begin{eqnarray*}
\gamma_2(g_1g_2) &=& \gamma_2\big((h_1n_1)(h_2n_2)\big)= \gamma_2(h_1h_2n_1^{h_2}n_2)\\
&=& \phi(h_1h_2)n_1^{h_2}n_2= \phi(h_1)\phi(h_2) n_1^{h_2}n_2\\
&=& \gamma_2(g_1)\gamma_2(g_2).
\end{eqnarray*}}
This shows that $\gamma_2$ is a homomorphism. It is not difficult to show that $\gamma_2$ is an automorphism of $G$ which centralizes $N$. Notice that $\psi_2$ is a section in the sequence \eqref{eqn7} and hence the sequence splits.
\para
Finally, we consider the sequence \eqref{eqn8}. Since $N$ is central, $G$ is a direct product of $H$ and $N$. Notice that $\C^* =  \Aut(N) \times \Aut(H)$. For a given pair $(\theta,\, \phi) \in \Aut(N) \times \Aut(H)$, we can define $f \in \Aut(G)$ by $f(hn)= \phi(h) \theta(n)$ for $g=hn$ in $G$. This gives rise to a section in the sequence \eqref{eqn8} and hence the sequence splits.  $\Box$
\end{proof}
\para
\begin{remark}
The converse of Theorem 8 is not true, in general, as is shown by the following examples.\para
\begin{enumerate}
\item Let $1 \to N \to G \to H \to 1$ be an exact sequence, where $G$ is a non-abelian finite group of nilpotency class 2 such that $N=Z(G)=[G,\,G]$ and $H \simeq G/N$. Notice that this sequence does not split under the natural action of $H$ on $N$. For, if the sequence splits, then $G$ is a direct product of $N$ and $H$. This implies that $G$ is abelian, which is a contradiction. In this case, $\Aut^{N,\, H}(G)= \Autcent(G)= \Aut^H(G)$, where $\Autcent(G)$ is the group of central automorphisms of $G$. Thus from the exactness of sequence \eqref{eqn6}, $\C_1^*=1$ and the sequence splits.
\\
\\
\item Consider an exact sequence $1 \to N \to G \to H \to 1$, where $G$ is an extra-special 2-group of order $2^{2n+1}$ with $n=1$ or $2$, and $N=Z(G)=[G,\,G]$. Notice that the sequence does not split. For this sequence we have $\Aut^{N,\, H}(G)=\Inn(G)=\Autcent(G)$ and $\Aut^N(G)=\Aut(G)$. Define a map $\rho:H \times H \to N$ by $\rho \big(t(x)N,\, t(y)N \big)= [t(x),\, t(y)]$. Notice that $\rho$ is a bilinear map. Let $\gamma \in \Aut(G)$ and $\phi = \tau_2(\gamma)$. Now
{\setlength\arraycolsep{2pt}
\begin{eqnarray*}
\rho \big(\phi(t(x)N),\, \phi(t(y)N) \big) &=& \rho \big(\gamma(t(x))N,\, \gamma(t(y))N \big)= [\gamma(t(x)),\, \gamma(t(y))]\\
&=& \gamma([t(x),\, t(y)])= [t(x),\, t(y)]\\
&=& \rho \big(t(x)N,\, t(y)N \big).
\end{eqnarray*}}
This shows that $\phi$,\, viewed as a linear transformation of the $\mathbb{F}_2$-vector space $H$, is orthogonal. Thus $\phi \in O(2n,\,2)$. This shows that $\C_2^* \subset O(2n,\,2)$. It is well-known that $\Aut(G)/\Inn(G)$ is isomorphic to the full orthogonal group $O(2n,\,2)$. Thus from the exactness of the sequence \eqref{eqn7}, we have $\C_2^* = O(2n,\,2)$. It follows from \cite[Theorem 1]{Griess} that the sequence \eqref{eqn7} splits.\\
\\
\item Let $1 \to N \to G \to H \to 1$ be the exact sequence of example (2) above. Since $\Aut(N)=1$, the equations \eqref{eqn7} and \eqref{eqn8} are same. Hence the sequence \eqref{eqn8} splits while $1 \to N \to G \to H \to 1$ does not split.
\end{enumerate}
\end{remark}

\Para

\bibliographystyle{amsplain}

\end{document}